\documentclass[a4paper,11pt,onecolumn]{amsart}
\usepackage{amsfonts}

\usepackage{amscd}
\usepackage{amsmath}
\usepackage{amssymb}
\usepackage{amsthm}
\usepackage{epsf}
\usepackage{latexsym}
\usepackage{verbatim}
\usepackage{color}
\input epsf.tex
\usepackage{bbm}

\usepackage{epstopdf}
\usepackage[pdftex]{graphicx}

\newtheorem{theorem}{Theorem}[section]

\newtheorem{lemma}[theorem]{Lemma}
\newtheorem{corollary}[theorem]{Corollary}

\theoremstyle{definition}

\begin{document} 

\title{Shapes of Polynomial Julia Sets}

\author[Kathryn A. Lindsey]{Kathryn A. Lindsey}
\address[Kathryn Lindsey]{ Department of Mathematics, Cornell University\\ Ithaca, NY 14853, USA}
\email{ klindsey@math.cornell.edu}


\subjclass{}
\keywords{}

\date{\today}
\maketitle 

\begin{abstract}
Any Jordan curve in the complex plane can be approximated arbitrarily well in the Hausdorff topology by Julia sets of polynomials.   Finite collections of disjoint Jordan domains can be approximated by the basins of attraction of rational maps.  \end{abstract}

  \begin{figure}[!h] 
  \centering
  \includegraphics[width=\linewidth]{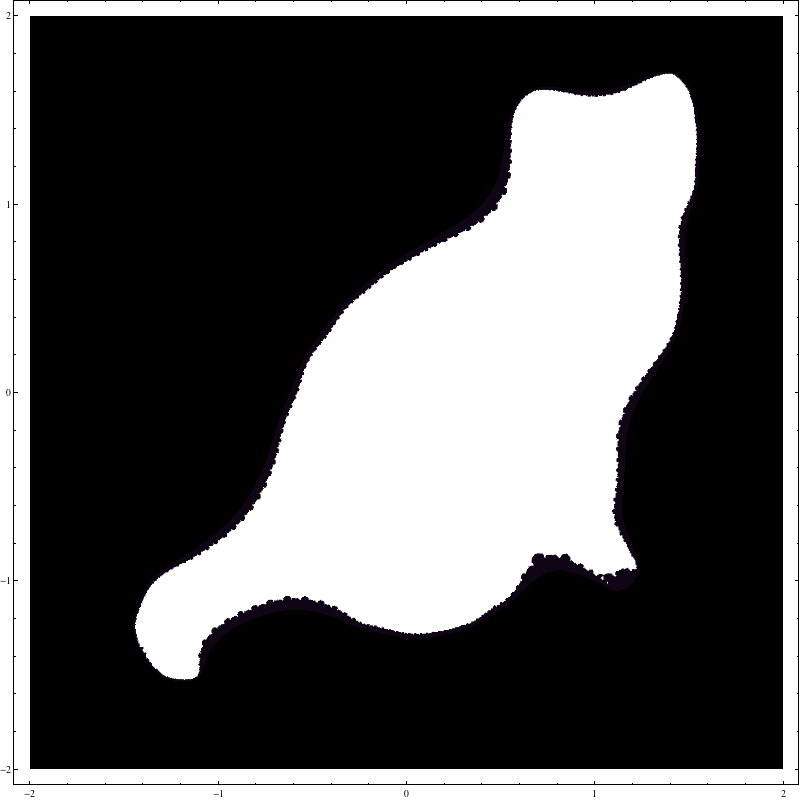}
  \caption[]{A filled Julia set in the shape of a cat.  The polynomial which generates this Julia set has degree 301.
  }
  \end{figure}
  
  For any complex-valued rational function $f$ on the Riemann sphere $\hat{\mathbb{C}}$ such that the point $\infty$ is an attracting fixed point of $f$, define the following notation:
  
  $$\begin{array}{lcl}
\mathcal{L}(f) &=& \textrm{the basin of attraction of }   \infty \textrm{ for the map } f,  \\
\mathcal{K}(f) &=& \hat{\mathbb{C}} \setminus \mathcal{L}(f), \\
\mathcal{J}(f) &=& \partial(\mathcal{K}(f)), \\
 \end{array}$$
 where $\partial(X)$ denotes the boundary of a set $X$.  If $f$ is a a polynomial then $\mathcal{J}(f)$ is the \emph{Julia set} of $f$, $\mathcal{K}(f)$ is the \emph{filled Julia set} of $f$, and $\mathcal{K}(f)$ is given by $$\mathcal{K}(f) = \{z \in \hat{\mathbb{C}} \mid f^n(z) \nrightarrow \infty \textrm{ as } n \rightarrow \infty\}.$$

 Theorem \ref{t:JuliaFitInAnnulus} gives a constructive proof that for any open bounded annulus $A$ in the complex plane, there exists a polynomial $P$ such that $\mathcal{J}(P)$ contains an essential closed curve of the closure $\bar{A}$, the bounded connected component of $\hat{\mathbb{C}} \setminus \bar{A}$ is contained in $\mathcal{K}(P)$, and the unbounded connected component of $\hat{\mathbb{C}} \setminus \bar{A}$ is contained in $\mathcal{L}(P)$.  Theorem \ref{t:HausdorffVersion} reformulates this result in terms of the Hausdorff metric, showing that any Jordan curve in the complex plane can be approximated arbitrarily well in the Hausdorff topology by Julia sets of polynomials.   Theorems \ref{t:RationalVersion}, \ref{t:HausdorffVersionRational} and \ref{t:HausdorffRationalAnnuliVersion} combine these polynomials to form rational maps whose basins of attraction approximate more complicated sets.  
  
   \subsection*{Acknowledgments}
   \medskip
     This project was initially a joint undertaking with William P. Thurston, who conjectured that Theorem \ref{t:HausdorffVersion} should be true and guided the author toward the approach presented in this paper.  Sadly, William Thurston passed away in August 2012, before this project reached fruition. 
     \medskip
  
The author wishes to express her gratitude to the anonymous referee, whose numerous helpful suggestions greatly improved the paper.   She also thanks John Smillie, John Hamal Hubbard and Sarah Koch for many helpful conversations during the development of this paper.  Kathryn Lindsey was supported by the Department of Defense through a National Defense Science and Engineering Graduate Fellowship and by the National Science Foundation through a Graduate Research Fellowship.  


\section{Annular decompositions and polynomial Julia sets}

   Define $\mathcal{A}$ to be the collection of all subsets of $\hat{\mathbb{C}}$ which can be written as the closure of an open, bounded (i.e. has empty intersection with some ball about $\infty$) annulus in $\hat{\mathbb{C}}$.   For any $A \in \mathcal{A}$, $\hat{\mathbb{C}} \setminus A$ consists of two topological disks; let $A_{\infty}$ denote the unbounded connected component of $\hat{\mathbb{C}} \setminus A$ and let $A_0$ denote the bounded connected component of $\hat{\mathbb{C}} \setminus A$.  Thus for each $A \in \mathcal{A}$, $\hat{\mathbb{C}} = A_{\infty} \sqcup A \sqcup A_0$.  
 
For each $A \in \mathcal{A}$, let $\gamma_A$ be a fixed Jordan curve in $\mathbb{C}$ contained in the interior of $A$ such that $\pi_1(\gamma_A)$ injects into $\pi_1(A)$ under the map induced by inclusion (i.e. $\gamma_A$ is an essential core curve of the annulus $A$).  Define $U_A$ to be the unbounded connected component of $\hat{\mathbb{C}} \setminus \gamma_A$, which is a topological disk.   Let  $\widetilde{\varphi_A}$ denote a fixed inverse Riemann map $$\widetilde{\varphi_A}: \{z \in \hat{\mathbb{C}} \mid 1 < |z| \} \rightarrow U_A$$ such that the points at infinity correspond. For some sufficiently small $\epsilon = \epsilon(A) >0$, the set $\widetilde{\varphi_A}(\{z \in \hat{\mathbb{C}} \mid |z|=1+\epsilon\}$ is contained in the interior of $A$.  Define $C_A$ to be this set, i.e. $$C_A = \widetilde{\varphi_A}(\{z \in \hat{\mathbb{C}} \mid |z|=1+\epsilon \}).$$ 
We will now rescale the conformal map $\widetilde{\varphi_A}$ so that the unit circle maps to $C_A$.  Specifically, define 
the map $\varphi_A :\{ z \in \hat{\mathbb{C}} \mid  \frac{1}{1+\epsilon} < |z| \} \rightarrow \hat{\mathbb{C}}$ by $$\varphi_A(z)=\widetilde{\varphi_A}\left((1+\epsilon)(z)\right).$$

\begin{lemma} \label{l:CAproperties} For each $A \in \mathcal{A}$, 
\begin{enumerate}
\item The set $C_A$ is a Jordan curve contained in the interior of $A$ such that $\pi_1(C_A)$ injects into $\pi_1(A)$ under the map induced by inclusion.  
\item $C_A$ is rectifiable with respect to the Euclidean metric on $\mathbb{C}$. 
\item $\varphi_A^{\prime}$ is non-vanishing and of bounded variation on the unit circle. 
\end{enumerate}
\end{lemma}

\begin{proof}
The set $C_A$ is a conformal image of $S^1$ and hence a Jordan curve.  It is homotopic to $\gamma_A$, whose fundamental group injects nontrivially by assumption, proving (i).  Since $\varphi_A$ is a conformal map, $\varphi^{\prime}$ is everywhere nonzero.  By the maximum modulus principle, $|\varphi_A^{\prime}|$ is bounded on $S^1$, which is itself of finite Euclidean arc length, implying (ii) and (iii).    
\end{proof}

For each $A \in \mathcal{A}$, the set $\hat{\mathbb{C}} \setminus C_A$ consists of two topological disks; let $O_A$ denote the unbounded connected component of $\hat{\mathbb{C}} \setminus C_A$ and let $I_A$ denote the bounded connected component of $\hat{\mathbb{C}} \setminus C_A$.  Thus for each $A \in \mathcal{A}$, we have a decomposition $\hat{\mathbb{C}} = O_A \sqcup C_A \sqcup I_A.$

Because $\varphi_A$ is injective, the coefficients for all the terms of degree greater than $1$ in the Laurent series centered at $0$ for $\varphi_A$ must be zero.  Hence any such Laurent series has the form 
$$\varphi_A(z) = c_Az + c_{(A,0)} + \frac{c_{(A,1)}}{z} + \frac{c_{(A,2)}}{z^2} + ... \ , \ \ \  |z| \geq 1,$$ with $c_A, c_{(A,0)}, c_{(A,1)},... \in \mathbb{C}$.  For each pair of natural numbers $(k,n)$ with $k \leq n$ and for each $A \in \mathcal{A}$, set 
$$r_A^{(k,n)} = \varphi_A(e^{k2\pi i/n}).$$

For each $n \in \mathbb{N}$ and $A \in \mathcal{A}$, define $\omega_{A,n}: \hat{\mathbb{C}} \rightarrow \hat{\mathbb{C}}$ by 

$$ \omega_{A,n}(z) = c_A^{-n}\prod_{j=1}^n (z-r_A^{(k,n)}). $$
For each $n \in \mathbb{N}$ and $A \in \mathcal{A}$, define the polynomial $P_{A,n}:\hat{\mathbb{C}} \rightarrow \hat{\mathbb{C}}$ by $$P_{A,n}(z) = z(\omega_{A,n}(z)+1).$$ 

The definitions of the map $\varphi$ and the polynomials $\omega_n$ use only the Jordan curve $C_A$, and not the fact that $C_A$ came from the annulus $A$.  Thus any Jordan curve $C$ in $\mathbb{C}$ determines both a conformal map $\varphi_C$ and a family of polynomials $\omega_{C,n}$, $n \in \mathbb{N}$; these polynomials are called \emph{fundamental polynomials of Lagrange interpolation}.  Curtiss proved the following two theorems about the uniform convergence of such polynomials in \cite{Curtiss} (we transcribe them here using our notation):

\begin{theorem} \label{t:Curtiss} \cite{Curtiss}
If $C$ is a rectifiable Jordan curve in $\mathbb{C}$, then 
\begin{equation} \label{eq:1.2} \lim_{n \rightarrow \infty} \omega_{C,n}(z) = -1\end{equation}
uniformly for $z$ on any closed subset of the region interior to $C$, and 
\begin{equation} \label{eq:1.3} \lim_{n \rightarrow \infty} \frac{\omega_{C,n}(z)}{w^n -1} = 1, \ \ \  \textrm{where } z=\varphi_C(w),\end{equation}
 uniformly for $z$ on any closed subset of the region exterior to $C$. 
\end{theorem}

\begin{theorem} \cite{Curtiss} \label{t:Curtiss2} If $\varphi_C^{\prime}(w)$ is non-vanishing and of bounded variation for $|w|=1$ then limit (\ref{eq:1.3}) in Theorem \ref{t:Curtiss} converges uniformly for $z$ in the closed region $C \sqcup C_{\infty}$, where $C_{\infty}$ denotes the unbounded connected component of $\hat{\mathbb{C}} \setminus C$.  
\end{theorem}

The reason for choosing the curve $C_A$ as we did was to use Curtiss' theorems to guarantee that the polynomials $\omega_{A,n}$ converge uniformly not only on bounded subsets of $O_A$ (as in Theorem \ref{t:Curtiss}) but also on unbounded compact subsets of $O_A \subset \hat{\mathbb{C}}$.  Lemma \ref{l:CAproperties} together with Theorems \ref{t:Curtiss} and \ref{t:Curtiss2} immediately yield:

\begin{corollary} \label{c:omegaBehavior}
For any $A \in \mathcal{A}$ and any compact set $I^{\prime} \subset I_A$  
$$\lim_{n \rightarrow \infty} \omega_{A,n}(z) = -1$$ uniformly for $z$ on $I^{\prime}$ and
$$ \lim_{n \rightarrow \infty} \frac{\omega_{A,n}(z)}{w^n -1} = 1, \ \ \  \textrm{where } z=\varphi_A(w)$$ uniformly for $z$ on $C_A \sqcup O_A$.
\end{corollary}

\begin{lemma} \label{l:lemma1}
Fix a real number $r >0$ and any $A \in \mathcal{A}$ such that $0 \in A_0$.  Then there exists a real number $K \geq 1$ such that for any real number $\kappa > K$ there exists $N \in \mathbb{N}$ such that $n \geq N$ implies 
\begin{enumerate}
\item $|P_{A,n}(z)| < r$ for all $z \in A_0$,
\item $|P_{A,n}(z)| > \kappa \cdot |z|$ for all $z \in A_{\infty}$, and 
\item $\kappa \cdot |z| > \sup\{|x| : x \in \hat{\mathbb{C}} \setminus A_{\infty} \}$ for all $z \in A_{\infty}$.
\end{enumerate}
\end{lemma}

\begin{proof}

Without loss of generality, we will assume that $r$ is small enough that  $B_r(0) \subset A_0$, where $B_r(0)$ denotes the open ball of radius $r$ about $0$.  Set 
$$\begin{array}{lcl}
\alpha &=& \sup \{|x|:x \in A_0\},  \\
\beta &=& \sup \{|x| : x \in A_0 \cup A \} \\
\gamma &=& \inf \{|x| : x \in A_{\infty} \} \\
\delta &=& \inf \{|x| : x \in \varphi_A^{-1}(A_{\infty}) \}.
 \end{array}$$
Notice that $\alpha$, $\beta$, and $\gamma$ are all at least $r$, and, in particular, are nonzero.  Notice also that $\delta > 1$. This is because the distance function from a point in the unit circle to the set $\varphi_A^{-1}(\partial A_{\infty}) \subset \{w:|w|>1\}$ is continuous on the unit circle, and hence attains its minimum, which is nonzero. Hence $\inf \{|x| : x \in \varphi_A^{-1}(A_{\infty}) \} > 1$.

Let $K = \max\{1, \frac{\beta}{\gamma}\}$.  Fix $\kappa > K$; then $\kappa>1$ and $\kappa \cdot \gamma > \beta$.  Pick $\epsilon_2 > 0$ such that $\epsilon_2 < r/\alpha$. Since $$\lim_{n\rightarrow \infty} \frac{\omega_{A,n}(z)}{(\varphi_A^{-1}(z))^n -1}=1$$ uniformly for $z \in A_{\infty}$ by Corollary \ref{c:omegaBehavior}, and $|\varphi_A^{-1}(z)| > 1$ on $ A_{\infty}$, we also have that  $$\lim_{n\rightarrow \infty} \frac{\omega_{A,n}(z)+1}{(\varphi_A^{-1}(z))^n -1}=1$$ uniformly on $A_{\infty}$.  Hence there exists $N_1 \in \mathbb{N}$ such that $n \geq N_1$ implies $$|\omega_{A,n}(z) +1| > (1-\epsilon_2) \cdot |w^n-1|$$ for all $z \in A_{\infty}$, with $z = \varphi_A(w)$.  
  Pick $N_2 \in \mathbb{N}$ such that $n \geq N_2$ implies $$\delta^n> 1 + \frac{\kappa}{1-\epsilon_2}.$$
  Since $\omega_{A,n}(z)$ converges uniformly to $-1$ on $A_0$ by Corollary \ref{c:omegaBehavior}, there exists $N_3 \in \mathbb{N}$ such that $n \geq N_3$ implies $|\omega_{A,n}(z) + 1| < \epsilon_2$ for all $z \in A_0$.   Let $N = \max \{N_1,N_2,N_3\}$. 
  
  Now consider $P_{A,n}$ for $n \geq N$.  For $z \in A_0$, $$|P_{A,n}(z)| = |z| \cdot |\omega_{A,n}(z)+1| \leq \alpha \cdot \epsilon_2 < r.$$ 
  For $z \in A_{\infty}$, 
  $$|P_{A,n}(z)| = |z| \cdot |\omega_{A,n}(z) +1| \geq |z| \cdot (1-\epsilon_2) \cdot |w^n+1| \geq |z| \cdot (1-\epsilon_2) \cdot | |w|^n - 1|  $$
  $$ \geq  |z| \cdot (1-\epsilon_2) \cdot (\delta^n -1) > |z| \cdot (1-\epsilon_2)(\frac{\kappa}{1-\epsilon_2}) = |z| \cdot \kappa$$
\end{proof}

\begin{theorem} \label{t:JuliaFitInAnnulus}
Fix $A \in \mathcal{A}$ and any point $t \in A_0$.  Let $\tau:\hat{\mathbb{C}} \rightarrow \hat{\mathbb{C}}$ denote the translation $z \mapsto z - t$.   Then there exists $N \in \mathbb{N}$ such that $n \geq N$ implies

\begin{enumerate}
\item $A_0 \subset \mathcal{K}(\tau^{-1} \circ P_{\tau(A),n} \circ \tau)$, 
\item $A_{\infty} \subset \mathcal{L}(\tau^{-1} \circ P_{\tau(A),n} \circ \tau)$,
\item $\mathcal{J}(\tau^{-1} \circ P_{\tau(A),n} \circ \tau ) \subset A$, and 
\item any path from a point in $\partial A_{\infty}$ to a point in $\partial A_0$ has nonempty intersection with $\mathcal{J}(\tau^{-1} \circ P_{\tau(A),n} \circ \tau)$. 
\end{enumerate}
\end{theorem}

\begin{proof}
To start, we will consider only the case that $0 \in A_0$ and $\tau$ is the identity map.   So pick $r>0$ such that $B_r(0) \subset A_0$.  Pick a real number $\kappa >1$ and $N \in \mathbb{N}$ as in the statement of Lemma \ref{l:lemma1}.  Let $n \geq N$, $n \in \mathbb{N}$.  

Since $A_0$ contains the disk $B_r(0)$, the property  $|P_{A,n}(z)| < r$ for all $ z \in A_0$ implies $A_0 \subset \mathcal{K}(P_{A,n})$.  The property $$|P_{A,n}(z)| >  \kappa \cdot |z| >  \sup \{|w| : w \in \hat{\mathbb{C}} \setminus A_{\infty} \}$$  for all  $z \in A_{\infty}$ implies that $P_{A,n}(A_{\infty}) \subset A_{\infty}$, and thus that $|P_{A,n}^{m}(z)| > \kappa^m \cdot |z|$ for all $m \in \mathbb{N}$ and $z \in A_{\infty}$.  Hence, $A_{\infty} \subset \mathcal{L}(P_{A,n})$.  

The partition of the Riemann sphere as $\hat{\mathbb{C}} = A_{\infty} \sqcup A \sqcup A_0$ together with the facts $A_0 \subset \mathcal{K}(P_{A,n})$ and  $A_{\infty} \subset \mathcal{L}(P_{A,n})$ imply that $\partial \mathcal{K}(P_{A,n}) \subset A$.  Any path connecting a point in $\mathcal{K}(P_{A,n})$ and a point in $(\mathcal{K}(P_{A,n}))^c = \mathcal{L}(P_{A,n})$ must have nonempty intersection with $\partial \mathcal{K}(P_{A,n})$.  This proves the theorem in the case that $0 \in A_0$ and $\tau$ is the identity.    

Now consider the general case, in which $\tau$ is any fixed map $z \mapsto z-t$ for some $t \in A_0$.  The translated annulus $\tau(A)$ is an element of $\mathcal{A}$ and its interior region, $\tau(A)_0$, contains the point $0$.  Hence the conclusion of the theorem holds for the annulus $\tau(A)$ and the identity translation.  The composition $\tau^{-1}\circ P_{\tau(A),n} \circ \tau$ is a polynomial and $\mathcal{K}(\tau^{-1}\circ P_{\tau(A),n} \circ \tau) = \tau^{-1}(\mathcal{K}(P_{\tau(A),n}))$.   Therefore $\mathcal{K}(\tau^{-1} \circ P_{\tau(A),n} \circ \tau) \supset \tau^{-1}(\tau(A)_0) = A_0$; statements (ii), (iii) and (iv) follow similarly. 

\end{proof}

\section{Hausdorff distance and annular approximations}

  \begin{figure}[!h] 
  \centering
  \includegraphics[width=\linewidth]{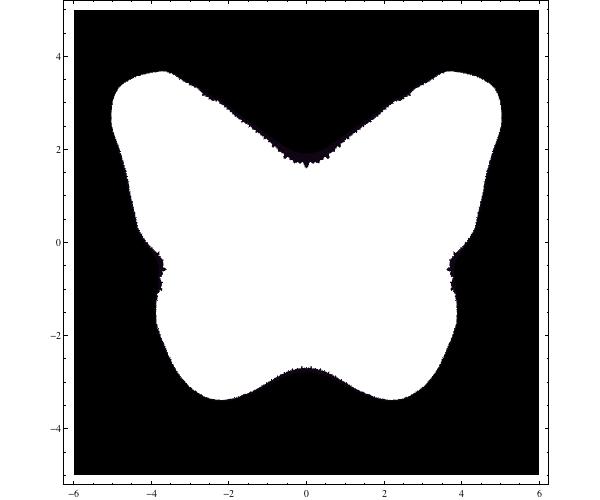}
  \caption[]{A filled Julia set in the shape of a butterfly.  The polynomial which generates this Julia set has degree 301.
  }
  \end{figure}

The purpose of this section is to carry out the technical arguments necessary to reformulate the conclusion of Theorem \ref{t:JuliaFitInAnnulus} in terms of the Hausdorff topology and Jordan curves (Theorem \ref{t:HausdorffVersion}).

 For any two non-empty subsets $X$ and $Y$ of $\mathbb{C}$, we denote by $d_H(X,Y)$ the Hausdorff distance between $X$ and $Y$, defined by $$d_H(X,Y) = \inf \{\epsilon > 0 \mid X \subseteq N_{\epsilon}(Y) \textrm{ and } Y \subseteq N_{\epsilon}(X)\},$$ where $$N_{\epsilon}(X) = \bigcup_{x \in X} \{z \in \hat{\mathbb{C}} \mid d(z,x) \leq \epsilon\}.$$  We will also consider the Hausdorff distance $d_H(X,Y)$ in the case where $X,Y \subset \hat{\mathbb{C}}$ and both $X$ and $Y$ contain an open ball, say $B$, about the point $\infty$.  In this case, we define $d_H(X,Y)$ to be the Hausdorff distance $d_H(X \setminus B, Y \setminus B)$, which is well-defined and finite.

We will denote by $\mathcal{C}$ the collection of Jordan curves in the complex plane $\mathbb{C} \subset \hat{\mathbb{C}}$.  The Jordan-Schoenflies Theorem says that for each Jordan curve $J \in \mathcal{C}$, there exists a homeomorphism $f_J:\mathbb{C} \rightarrow \mathbb{C}$ such that $f_J(J) = S^1$ and such that $f$ maps the bounded component of $J^c$ to the unit disk.  For each real number $\epsilon$ with $1 > \epsilon>0$, define 
$$J^{\epsilon} = f_J^{-1}(\{z \in \mathbb{C} \mid 1-\epsilon \leq  |z| \leq 1+\epsilon \}).$$  
The set  $\hat{\mathbb{C}} \setminus J^{\epsilon}$ consists of two connected components; let $J^{\epsilon}_0$ denote the bounded component and let $J^{\epsilon}_{\infty}$ denote the unbounded component.  Thus for each $J \in \mathcal{C}$ and $\epsilon > 0$ we have a decomposition $\hat{\mathbb{C}} = J^{\epsilon}_0 \sqcup J^{\epsilon} \sqcup J^{\epsilon}_{\infty}$.  

\begin{lemma} \label{l:approxbyannuli}
Fix a Jordan curve $J \in \mathcal{C}$ and a real number $\delta >0$.  Then there exists a real number $\rho > 0$ such that $\rho > \epsilon > 0$ implies 
\begin{enumerate}
\item $J^{\epsilon} \subset N_{\delta}(J)$, 
\item For every point $j \in J$, $B_{\delta}(j) \cap J^{\epsilon}_0 \not = \emptyset$ and $B_{\delta}(j) \cap J^{\epsilon}_{\infty} \not = \emptyset$.
\end{enumerate}
\end{lemma}

\begin{proof}
Define a function $w_{\delta}:S^1 \rightarrow \mathbb{R}$ by 
$$w_{\delta}(s) = \sup \{r>0 \mid B_r(s) \subset f_J(N_{\delta}(J))\}.$$
Since $f_J(N_{\delta}(J))$ is an open set which contains $S^1$,  $w_{\delta}(s)$ is well-defined and positive for each $s \in S^1$.  Since $w_{\delta}$ is continuous on $S^1$ and $S^1$ is compact, there exists a real number $ \rho_1 > 0$ such that $w_{\delta}(s) \geq \rho_1$ for all $s \in S^1$.  Then for any real number $\epsilon$ with $0 < \epsilon < \rho_1$ and for any $s \in S^1$, we have 
$B_{\epsilon}(s) \subset f_J(N_{\delta}(J))$, and thus $$N_{\epsilon}(S^1) = \bigcup_{s \in S^1} B_{\epsilon}(s)  \subset f_J(N_{\delta}(J)).$$ Therefore $0 < \epsilon < \rho_1$ implies $J^{\epsilon} \subset N_{\delta}(J)$.

Define a function $y_{\delta}:J \rightarrow \mathbb{R}$ by 
$$y_{\delta}(j) =  \sup \{\epsilon > 0 \mid B_{\delta}(j) \cap J^{\epsilon}_{\infty} \not = \emptyset \textrm{ and } B_{\delta}(j) \cap J^{\epsilon}_0 \not = \emptyset  \}.$$  Notice that $y_{\delta}(j)$ is well-defined and positive for each point $j \in J$.  Since $y_{\delta}$ is continuous on $J$ and $J$ is compact ($J$ is the image of $S^1$ under $f^{-1}$), there exists a real number $\rho_2 > 0$ such that $y_{\delta}(j) > \rho_2$ for all $j \in J$.  Then $\rho_2 > \epsilon > 0$ implies $B_{\delta}(j) \cap J_{\infty}^{\epsilon} \not = \emptyset$ and $B_{\delta}(j) \cap J_0^{\epsilon} \not = \emptyset$ for all $j \in J$.  

Let $\rho = \min\{\rho_1, \rho_2\}$.  Then $\rho>\epsilon > 0$ implies properties (i) and (ii). 
\end{proof}

\begin{theorem} \label{t:HausdorffVersion}
Let $J$ be a Jordan curve in $\mathbb{C}$ and fix a real number $\delta > 0$.   Denote by $I$ the bounded component of $\hat{\mathbb{C}}\setminus J$ and denote by $O$ the unbounded component of $\hat{\mathbb{C}}\setminus J$.  There exists a polynomial $P:\hat{\mathbb{C}} \rightarrow \hat{\mathbb{C}}$ such that 
\begin{enumerate}
\item $d_H(I, \mathcal{K}(P)) < \delta$,
\item $d_H(J, \mathcal{J}(P)) < \delta$, and 
\item $d_H(O, \mathcal{L}(P)) < \delta$. 
\end{enumerate}
\end{theorem}
 
\begin{proof}
Pick a real number $\delta_1$ such that $0< \delta_1 < \delta/3$.  Apply Lemma \ref{l:approxbyannuli} using the parameter $\delta_1$ to obtain $\rho>0$ such that $\rho > \epsilon > 0$ implies 
\begin{enumerate}
\item[(a)]  $J^{\epsilon} \subseteq N_{\delta_1}(J)$ and 
\item[(b)]  $B_{\delta_1}(j) \cap J^{\epsilon}_0 \not = \emptyset$ and $B_{\delta_1}(j) \cap J^{\epsilon}_{\infty} \not = \emptyset$ for all $j \in J$.
\end{enumerate}
Fix a real number $\epsilon$ with $\rho > \epsilon >0$.  Apply Theorem \ref{t:JuliaFitInAnnulus} to the annulus $J^{\epsilon} \in \mathcal{A}$  to obtain a polynomial $P_n$ such that 
\begin{itemize}
\item $J^{\epsilon}_0 \subset \mathcal{K}(P_n)$,
\item $J^{\epsilon}_{\infty} \subset \mathcal{L}(P_n)$, 
\item $\mathcal{J}(P_n) \subset J^{\epsilon}$, 
\item any path from a point in $\partial J^{\epsilon}_{\infty}$ to a point in $\partial J^{\epsilon}_0$ has nonempty intersection with $\mathcal{J}(P_n)$.  
\end{itemize}

We know $J^{\epsilon}_0 \subseteq \mathcal{K}(P_n) \subseteq J^{\epsilon} \sqcup J^{\epsilon}_0$ and $J^{\epsilon}_0 \subseteq I \subseteq J^{\epsilon} \sqcup J^{\epsilon}_0$.  Consequently the symmetric difference $\mathcal{K}(P_n) \bigtriangleup I$ is contained in $J^{\epsilon}$.  By (a) any point $x \in J^{\epsilon}$ is within distance at most $\delta_1$ of some point $j \in J$, and by (b), the point $j$ is within distance at most $\delta_1$ of a point in $J_0^{\epsilon} \subset I \cap \mathcal{K}(P_n)$.  Hence $d_H(I,\mathcal{K}(P_n)) \leq 2\delta_1 < \delta.$  

Similarly, $J^{\epsilon}_{\infty} \subseteq \mathcal{L}(P_n) \subseteq J^{\epsilon} \sqcup J^{\epsilon}_{\infty}$ and $J^{\epsilon}_{\infty} \subseteq O \subseteq J^{\epsilon} \sqcup J^{\epsilon}_{\infty}$, implying $\mathcal{L}(P_n) \bigtriangleup O \subset J^{\epsilon}$ and thus $d_H(\mathcal{L}(P_n),O) < \delta$. 

Let $j$ be a point in $J$.  By (b), pick points  $p_1 \in B_{\delta_1}(j) \cap J^{\epsilon}_0$ and $p_2 \in B_{\delta_1}(j) \cap J^{\epsilon}_{\infty}$ and let $\gamma$ be the path consisting of the two line segments $\overline{p_1j}$ and $\overline{p_2j}$.  By construction, the length of $\gamma$ is at most $2\delta_1$.  Since $\gamma$ connects $\mathcal{K}(P_n)$ with $\mathcal{L}(P_n) = \mathcal({K}(P_n))^c$, $\gamma$ contains a point in $\partial(\mathcal{K}(P_n)) = \mathcal{J}(P_n)$.  Hence $d_H(j,\mathcal{J}(P_n)) \leq 2\delta_1$.  The containments 
$$\mathcal{J}(P_n) \subset J^{\epsilon} \subset N_{\delta_1}(J)$$
imply $d_H(a,J) \leq \delta_1$ for any point $a \in \mathcal{J}(P_n)$.  Therefore $$d_H(\mathcal{J}(P_n),J) < \delta.$$
\end{proof}

\section{Rational Maps}

  \begin{figure}[!h] 
  \centering
  \includegraphics[width=\linewidth]{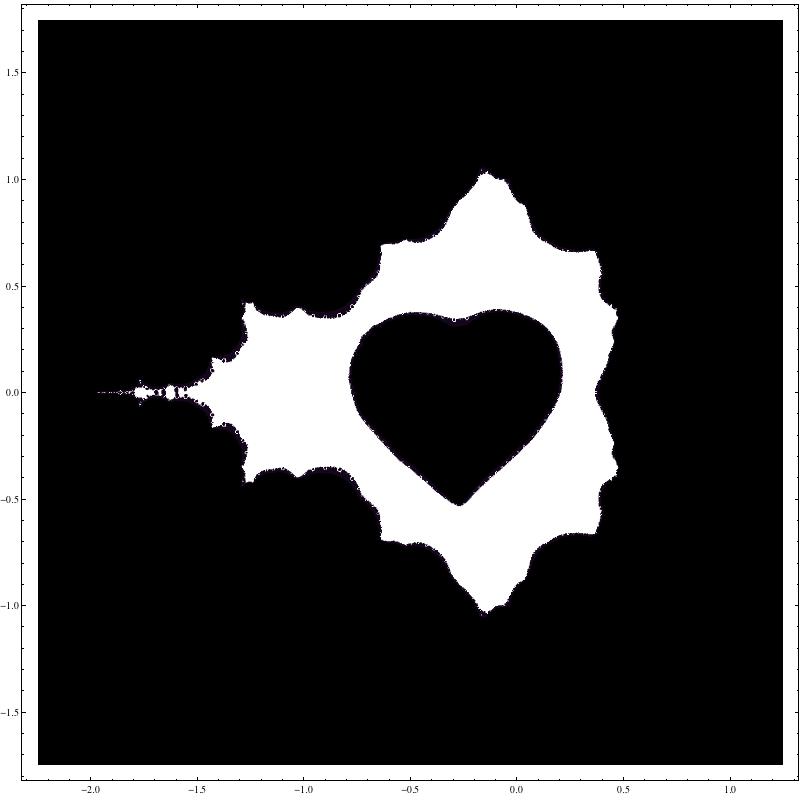}
  \caption[]{$\mathcal{K}(\rho)$ for a rational map $\rho$ constructed as in Theorem \ref{t:RationalVersion}.  }
  \end{figure}
  
  Let $\mathfrak{A}$ denote a finite collection $\{A_1,...,A_m\}$ of annuli in $\mathcal{A}$.  For each $i$, let $A_{i,0}$ and $A_{i,\infty}$ denote the bounded component and unbounded component, respectively, of $\hat{\mathbb{C}} \setminus A_i$.   Let $|\mathfrak{A}|$ denote the number of elements of $\mathfrak{A}$.  We will say that the elements of $\mathfrak{A}$ are \emph{mutually exterior} if $A_j \cap (A_i \sqcup A_{i,0}) = \emptyset$ for all $i \neq j$. 
  
  For any finite collection $\mathfrak{A}$ of mutually exterior annuli in $\mathcal{A}$ and any $n \in \mathbb{N}$, define a rational function $\Omega_{\mathfrak{A},n}:\hat{\mathbb{C}} \rightarrow \hat{\mathbb{C}}$ by 
  \begin{equation}  \label{eq:Omega}
\Omega_{\mathfrak{A},n}(z) = \left(
\sum_{j=1}^{|\mathfrak{A}|} (\omega_{A_j,n}(z) + 1)^{-1}
 \right)^{-1}.
  \end{equation}
 Also define for each $n \in \mathbb{N}$ the rational function $R_{\mathfrak{A},n}:\hat{\mathbb{C}} \rightarrow \hat{\mathbb{C}}$ by $$R_{\mathfrak{A},n}(z) = z \cdot \Omega_{\mathfrak{A},n}(z).$$
  Define also 
    $$\begin{array}{lcl}
  \mathfrak{A}_+ &=& \bigcup_{j=1}^{|\mathfrak{A}|} A_{j}, \\
  \mathfrak{A}_0 &=& \bigcup_{j=1}^{|\mathfrak{A}|} A_{j,0}, \\
  \mathfrak{A}_{\infty} &=& \bigcap_{j=1}^{|\mathfrak{A}|} A_{j,\infty}. \\
  \end{array}$$
  Thus we have a decomposition $\hat{\mathbb{C}} = \mathfrak{A}_0 \sqcup \mathfrak{A}_+ \sqcup \mathfrak{A}_{\infty}$.  
  
  \begin{lemma} \label{l:controlOnLotsOfAnnuli}
  Let $\mathfrak{A}$ be a finite collection of mutually exterior annuli in $\mathcal{A}$ and fix real numbers $B > b > 0$.  Then there exists $N \in \mathbb{N}$ such that for all integers $n \geq N$, 
  \begin{enumerate}
  \item $|\Omega_{\mathfrak{A},n}(z)| < b$ for all $z \in \mathfrak{A}_0$, and
  \item $|\Omega_{\mathfrak{A},n}(z)| > B$ for all $z \in \mathfrak{A}_{\infty}$.  
  \end{enumerate}
  \end{lemma}
  
  \begin{proof}
  Pick positive real numbers $d$ and $D$ with $d$ small enough and $D$ large enough that $$\left|\frac{1}{d} - \frac{|\mathfrak{A}|-1}{D}\right|^{-1} < b \ \  \textrm{ and } \ \  \frac{D}{|\mathfrak{A}|} > B.$$
  As a consequence of Corollary \ref{c:omegaBehavior}, there exists $N \in \mathbb{N}$ such that for all $n \geq N$ and all $1 \leq j \leq |\mathfrak{A}|$, $|\omega_{A_j,n} (z)+ 1| < d$ for all $z \in A_{j,0}$ and $|\omega_{A_j,n} (z)+ 1| > D$ for all $z \in A_{j,\infty}$. 
  Fix $n \geq N$.  For $z \in A_{i,0}$,
$$\left| \sum_{j=1}^{|\mathfrak{A}|} \frac{1}{\omega_{A_j,n}(z)+1}  \right|  \geq 
\left| \left| \frac{1}{\omega_{A_i,n}(z)+1} \right| - \sum_{j=1,j \not = i}^{|\mathfrak{A}|} \left| \frac{1}{\omega_{j,n}(z)+1} \right| \right| \geq \left| \frac{1}{d}-\frac{|\mathfrak{A}|-1}{D} \right|. $$
Hence $|\Omega_{\mathfrak{A},n}(z)| < b$ for all $z \in \mathfrak{A}_0$.   For $z \in \mathfrak{A}_{\infty}$,
 $$\left| \sum_{j=1}^{|\mathfrak{A}|} \frac{1}{\omega_{A_j,n}(z)+1}  \right| \leq \sum_{j=1}^{|\mathfrak{A}|} \left |\frac{1}{\omega_{A_j,n}(z)+1} \right| \leq \frac{|\mathfrak{A}|}{D}.$$ Hence $|\Omega_{\mathfrak{A},n}(z)| > B$ for all $z \in \mathfrak{A}_{\infty}$. 
  \end{proof}
  
  \begin{theorem} \label{t:RationalVersion}
  Let $\mathfrak{A}=\{A_1,...,A_m\}$ be a finite collection of mutually exterior annuli in $\mathcal{A}$.  Fix any point $t \in \mathfrak{A}_0$ and let $\tau:\hat{\mathbb{C}} \rightarrow \hat{\mathbb{C}}$ be the translation $z \mapsto z-t$.

   Then there exists $N \in \mathbb{N}$ such that for any integer $n \geq N$, 
  \begin{enumerate}
  \item $\mathfrak{A}_0 \subset \mathcal{K}(\tau^{-1} \circ R_{\tau(\mathfrak{A}),n} \circ \tau)$
  \item $\mathfrak{A}_{\infty} \subset \mathcal{L}(\tau^{-1} \circ R_{\tau(\mathfrak{A}),n} \circ \tau)$
  \item $ \mathcal{J}(\tau^{-1} \circ R_{\tau(\mathfrak{A}),n} \circ \tau) \subset \mathfrak{A}_+$
  \item For any $A_i \in \mathcal{A}$, any path connecting $A_{i,\infty}$ and $A_{i,0}$ has nonempty intersection with $\mathcal{J}(\tau^{-1} \circ R_{\tau(\mathfrak{A}),n} \circ \tau)$.
  \end{enumerate}
  \end{theorem}
  
  \begin{proof}
  To begin, we will consider the case that $0 \in \mathfrak{A}_0$ and $\tau$ is the identity map.  So assume $0 \in \mathfrak{A}_0$ and let $\rho$ be the radius of a small ball centered at $0$ and contained in $\mathfrak{A}_0$.  Fix a real number $b > 0$ small enough that 
  $$b \cdot \sup \{|z| \mid z \in \mathfrak{A}_0\} < \rho.$$  Fix a real number $B > 1$ large enough that 
  $$B \cdot \inf \{|z| \mid z \in \mathfrak{A}_{\infty}\} > \sup\{|z| \mid z \in \mathfrak{A}_+ \cup \mathfrak{A}_0\}.$$
  
  Let $N$ be the integer guaranteed by Lemma \ref{l:controlOnLotsOfAnnuli}, and fix an integer $n \geq N$.  Then for any $z \in \mathfrak{A}_0$
  $$|R_{\mathfrak{A},n}(z)| = |z| \cdot | \Omega_{\mathfrak{A},n}(z)| \leq |z| \cdot b < \rho.$$

    Hence $R_{\mathfrak{A},n}(z) \in \mathfrak{A}_0$ for any $z \in \mathfrak{A}_0$. 
  For any $z \in \mathfrak{A}_{\infty}$ 
  $$ |R_{\mathfrak{A},n}(z)| = |z| \cdot |\Omega_{\mathfrak{A},n}(z)| \geq |z| \cdot B.$$  Hence for any $z \in \mathfrak{A}_{\infty}$, we have $R_{\mathfrak{A},n}(z) \in \mathfrak{A}_{\infty}$ and so $|R^m_{\mathfrak{A},n}(z) | > B^m |z|$ for all $m \in \mathbb{N}$. 
  Therefore, the basin of attraction of $\infty$ for the map $R_{\mathfrak{A},n}$ contains all of $\mathfrak{A}_{\infty}$ and has empty intersection with $\mathfrak{A}_0$.  This proves the theorem in the case that $0 \in \mathfrak{A}_0$ and $\tau$ is the identity map. 
  
  For the general case, in which $\tau$ is any fixed map $z \mapsto z-t$ for some $t \in \mathfrak{A}_0$, it suffices to observe that $\mathcal{K}(\tau^{-1} \circ R_{\tau(\mathfrak{A}),n} \circ \tau) = \tau^{-1}(\mathcal{K}(R_{\mathfrak{A},n}))$.
  \end{proof}
  
  For any finite collection $\mathfrak{J}=\{J_1,...,J_m\}$ of mutually exterior bounded Jordan curves in $\hat{\mathbb{C}}$, define the sets $\mathfrak{J}_+$, $\mathfrak{J}_0$, and $\mathfrak{J}_{\infty}$ as follows.  Let $\mathfrak{J}_+ = \cup_i J_i$.  Let $\mathfrak{J}_0$ denote the union of the bounded connected components of $\hat{\mathbb{C}}\setminus J_i$.  Let $\mathfrak{J}_{\infty} = \hat{\mathbb{C}} \setminus (\mathfrak{J}_+ \cup \mathfrak{J}_0)$.  Thus $\hat{\mathbb{C}} = \mathfrak{J}_0 \sqcup \mathfrak{J}_+ \sqcup \mathfrak{J}_{\infty}$. 
  
  \begin{theorem}  \label{t:HausdorffVersionRational}
  Let $\mathfrak{J}=\{J_1,...,J_m\}$ be a finite collection of mutually exterior bounded Jordan curves in $\hat{\mathbb{C}}$ and fix a real number $\delta>0$.  Then there exists a rational map $Q:\hat{\mathbb{C}} \rightarrow \hat{\mathbb{C}}$ for which $\infty$ is an attracting fixed point and 
  \begin{enumerate}
  \item $d_H(\mathfrak{J}_0,\mathcal{K}(Q)) < \delta$,
  \item $d_H(\mathfrak{J}_+,\mathcal{J}(Q)) < \delta$, and
  \item $d_H(\mathfrak{J}_{\infty},\mathcal{L}(Q))<\delta$.
  \end{enumerate}
  \end{theorem}
  
  \begin{proof}
  Let $$\eta = \frac{1}{2} \inf \{ d(x,y) \mid x \in J_{i,0}, y \in J_{j,0}, i \not = j\}.$$  Because the curves in $\mathfrak{J}$ are mutually exterior, $\eta >0$.    Let $\delta_1 = \min\{\delta/3, \eta\}$.   For each $i$, let $\rho_i$ be the positive real number guaranteed by Lemma \ref{l:approxbyannuli} using the parameter $\delta_1$.   Let $\rho = \min\{\rho_0,\rho_1,...,\rho_m\}$.  
  
  Fix a real number $\epsilon$ such that $\rho > \epsilon >0$.  Let $\mathfrak{A}$ denote the collection of annulus $J^{\epsilon}_1,...,J^{\epsilon}_m$ (the notation $J^{\epsilon}$ is as defined at the beginning of Section 2).  The condition that $\delta_1 < \eta$ implies that the annuli in $\mathfrak{A}$ are all pairwise disjoint and thus mutually exterior.  Apply Theorem \ref{t:HausdorffVersionRational} using the finite collection of mutually exterior annuli $\mathfrak{A}$ to obtain a rational map $Q:\hat{\mathbb{C}} \rightarrow \hat{\mathbb{C}}$ such that 
  \begin{enumerate}
 \item $\mathfrak{A}_0 \subset \mathcal{K}$, 
\item  $\mathfrak{A}_{\infty} \subset \mathcal{L}(Q)$, 
\item $\mathcal{J}(Q) \subset \mathfrak{A}_+$, and
\item every path between a point in $\mathfrak{A}_0$ and a point in $\mathfrak{A}_{\infty}$ has nonempty intersection with $\mathcal{J}(Q)$.  
  \end{enumerate}
The Hausdorff distance estimates can then be proved in the same way as in the proof of  Theorem  \ref{t:HausdorffVersion}.
  \end{proof}
  
  \begin{theorem} \label{t:HausdorffRationalAnnuliVersion}
  Let $A$ be an annulus in $\mathcal{A}$ and fix a real number $\delta >0$.  Then there exists a rational function $S:\hat{\mathbb{C}} \rightarrow \hat{\mathbb{C}}$ for which $\infty$ is an attracting fixed point and 
  \begin{enumerate}
  \item $d_H(A,\mathcal{K}(S)) < \delta$,
  \item $d_H(A_{\infty} \sqcup A_0, \mathcal{L}(S)) < \delta$,
  \item $d_H(\partial A, \mathcal{J}(S)) < \delta$.
  \end{enumerate}
  \end{theorem}
  
  \begin{proof}

  The boundary of $A$ is composed of two Jordan curves; let $C_1$ denote the ``outer" curve and $C_2$ denote the inner curve.  Set 
  $$ \xi = \inf \{d(x,y) \mid x \in C_1, y \in C_2\}.$$
  Notice the compactness of the Jordan curves $C_1$ and $C_2$ implies $\xi>0$.  Let $E$ and $F$ be the annuli guaranteed by Lemma \ref{l:approxbyannuli} using $C_1$ and $C_2$, respectively, and the parameter $\delta_1 =  \frac{1}{3}\min\{\delta,\xi\}$.   By construction, $E$ and $F$ are disjoint.  Thus $\hat{\mathbb{C}} \setminus (E \cup F)$ consists of three connected components; denote the ``outside" component (which contains $\infty$) by $O$, denote the ``middle" component by $M$, and denote the ``inside" component by $I$.  Thus $\hat{\mathbb{C}} = O \sqcup E \sqcup M \sqcup F \sqcup I$.  
  
Assume for now that $0 \in M$.  Fix real numbers $R>r>0$ such that $B_r(0) \subset M$ and $(\hat{\mathbb{C}} \setminus O) \subset B_R(0)$.  
Set $$\eta = \frac{1}{ \inf\{|z| : z \in O\} }.$$   Fix a real number $\kappa$ such that 
$$\kappa > \max \left\{1, \eta (1+\frac{r}{2}), \eta (R+ \frac{r}{2})\right\}.$$
By Lemma \ref{l:lemma1}, there exists $N_E \in \mathbb{N}$ such that $n > N_E$ implies 
$$\left|P_{E,n} (z)\right| < \frac{r}{2} \ \ \ \textrm{ for all } z \in M \sqcup F \sqcup I$$ and 
$$|P_{E,n}(z)| > \kappa |z|  \ \ \  \textrm{for all } z \in O.$$
Notice that $$\kappa |z| > \eta (R + \frac{r}{2}) |z| \geq R+ \frac{r}{2}$$ for all $z \in O$.  Notice also that 
\begin{equation} \label{eq:bigenough}
\kappa |z| - \frac{r}{2} > (2 + \frac{r}{2} \eta) |z| - \frac{r}{2} = 2|z| + \frac{r}{2}(\eta |z| -1)  \geq 2 |z|
\end{equation} for all $z \in O$.

By Corollary \ref{c:omegaBehavior}, there exists $N_F \in \mathbb{N}$ such that for any integer $m>N_F$, 

$$\left| \frac{1}{\omega_{F,m}(z)+1} \right| > R + \frac{r}{2} \ \ \ \textrm{   for all } z \in I$$ 
and
 $$\left| \frac{1}{\omega_{F,m}(z)+1} \right| < \frac{r}{2} \ \ \ \textrm{ for all } z \in M \cup E \cup O.$$

For each $n \in \mathbb{N}$, define a rational map $S_n:\hat{\mathbb{C}} \rightarrow \hat{\mathbb{C}}$ by $$S_n(z) = P_{E,n}(z) + \frac{1}{\omega_{F,n}(z) + 1} .$$

Set $N = \max\{N_E, N_F\}$.  Fix any integer  $n > N$. The inequalities above imply that $|S_n(z)| >R$ for all $z \in O \sqcup I$ and $|S_n(z)| < r$ for all $z \in M$.   Consequently, $S_n(z) \in O$ for all $z \in O \sqcup I$ and $S_n(z) \in M$ for all $z \in M$.  Furthermore, for any $z \in O$ the inequality (\ref{eq:bigenough}) implies 
$$|S_n(z)| \geq \left| |P_{E,n}(z)| - |\frac{1}{\omega_{F,n}(z)+1}| \right|  \geq \kappa |z| - \frac{r}{2} > 2|z|.$$  Hence $|S_n^t(z)| > 2^t|z|$ for all $z \in O$ and $t \in \mathbb{N}$.  
Consequently, $O \sqcup I \subset \mathcal{L}(S_n)$ and $M \subset \mathcal{K}(S_n)$.  The Hausdorff distance estimates may then be obtained in a way similar to that in the proof of Theorem \ref{t:HausdorffVersion}.

  \end{proof}
\section{Discussion}

 Theorem \ref{t:HausdorffVersion} extends the analogy between the dynamics of rational maps and Kleinian groups.   The discussion of inversive two-manifolds in \cite{SullivanThurston} examines Kleinian groups generated by inversions through circles arranged in a chain along a Jordan curve.  The approach in \cite{SullivanThurston} leads to the following theorem:
\begin{theorem}  \label{t:KleinianCase} Any simple closed curve $C$ can be approximated in the Hausdorff topology by limits sets of finitely-generated quasi-Fuchsian groups.
 \end{theorem}
This result above is stated in \cite{HubbardV2}.  The maps constructed in this paper and in Theorem \ref{t:KleinianCase} all use a set of special points strung along the curve -- in the first case it is the roots of the polynomials, and in the second case it is the centers of the circles which define the inversions.

 \bigskip
 
 \noindent \textbf{Question.} Characterize those sets $S \subset \mathbb{C}$ such that for any $\epsilon>0$ there exists a polynomial (or rational map) $P$ such that $d_{Haus}(S,\mathcal{K}(P)) < \epsilon$ and $d_{Haus}(\hat{\mathbb{C}} - S,\hat{\mathbb{C}} - \mathcal{K}(P))<\epsilon$.  
\bigskip

The theme of this paper has been to first fix the ``shape" we wish to approximate, and then vary the degree of the polynomials whose Julia sets approximate the shape.  However, the reverse approach  -- first fix a degree $d$, and then analyze the set of Julia sets for polynomials of degree $d$ -- also warrants  investigation.  
 
 \bigskip
 
\noindent \textbf{Question.} Find a method to determine $$\{P: P \textrm{ is a polynomial of degree } d \textrm{ and } d_{Haus}(\mathcal{J}(P), S)  \leq  \epsilon\}$$ for any given $d \in \mathbb{N}$, $ \epsilon \geq 0$, and $S \subset \mathbb{C}$.

\bibliographystyle{amsalpha}
\bibliography{JuliaShapesBibliography}

\nocite{*}

\end{document}